\documentclass[11pt]{article}

\usepackage[utf8]{inputenc}

\usepackage{geometry}
\usepackage{amsthm,amsmath,amssymb}

\usepackage{tikz}

\theoremstyle{plain}
\newtheorem{theorem}{Theorem}
\newtheorem{lemma}[theorem]{Lemma}
\newtheorem{corollary}[theorem]{Corollary}

\title{Hamiltonian-connectedness of triangulations with few separating triangles}
\author{Nico Van Cleemput\\
\small Department of Applied Mathematics, Computer Science and Statistics,\\
\small Ghent University, Krijgslaan 281 - S9 - WE02, 9000 Ghent, Belgium\\
\small\tt nico.vancleemput@gmail.com
}

\begin{document}

\maketitle

\begin{abstract}
We prove that 3-connected triangulations with at most one separating triangle are hamiltonian-connected. In order to show bounds on the strongest form of this theorem, we proved that for any \(s\geq4\) there are 3-connected triangulation with \(s\) separating triangles that are not hamiltonian-connected. We also present computational results which show that all `small' 3-connected triangulations with at most 3 separating triangles are hamiltonian-connected.

\end{abstract}

\section{Introduction}

Plane triangulations -- sometimes also called maximal planar graphs -- are plane graphs where all faces, including the outer face, are triangles. A hamiltonian cycle, resp. path, is a cycle, resp. path, that visits each vertex exactly once. A graph is hamiltonian if it contains a hamiltonian cycle. A graph is hamiltonian-connected if for every pair of vertices there exists a hamiltonian path having these two vertices as end points. Clearly, being hamiltonian-connected implies being hamiltonian.

In 1931, Whitney \cite{Wh:31} proved that every 4-connected plane triangulation is hamiltonian. The condition of being 4-connected was later relaxed by Chen who proved that 3-connected triangulations with only one separating triangle are hamiltonian \cite{Ch:03}. An even stronger result of this form is given by Jackson and Yu \cite{JY:02}. They show that a nonhamiltonian triangulation must contain at least 4 separating triangles and even then they have to be in a specific configuration.

A stronger result of a different form was proved by Thomassen who showed that every 4-connected plane graph is hamiltonian-connected \cite{Th:83}. Ozeki and Vr\'ana \cite{OV:14} recently proved an even stronger result in function of a property called \emph{\(k\)-edge-hamiltonian-connected}. A graph \(G\) is \(k\)-edge-hamiltonian-connected if for any \(X \subset \{x_1x_2 : x_1, x_2 \in V(G), x_1 \neq x_2 \}\) such that \(1 \leq |X| \leq k\) and the graph
induced by \(X\) on \(V(G)\) is a forest in which each component is a path, \(G \cup X\) has a hamiltonian cycle containing all edges in \(X\), where \(G \cup X\) is the graph obtained from \(G\) by adding all edges in \(X\). Ozeki and Vr\'ana proved the following theorem.

\begin{theorem}[\cite{OV:14}(3)]\label{thm:OZ:14}
Every 4-connected plane graph is 2-edge-hamiltonian-connected.
\end{theorem}

In Section~\ref{sec:onesep}, we will show that 3-connected triangulations with only one separating triangle are hamiltonian-connected. We will do this by showing a stronger result about 3-connected triangulations having an edge that is contained in all separating triangles. In Section~\ref{sec:non_hc} and Section~\ref{sec:computational} we investigate how much stronger this theorem could be made using an investigation similar to the one in \cite{BSVC:2015}. We first show that starting with 4 separating triangles there always exist 3-connected triangulations that are not hamiltonian-connected, and we further specify this in terms of the decomposition tree of a 3-connected triangulation. Finally, in Section~\ref{sec:computational} we give the results of some computer searches to check the hamiltonian-connectedness of 3-connected triangulation with a small number of separating triangles.

\section{One separating triangle}\label{sec:onesep}

The theorem which we will prove here is mainly a corollary of Theorem~\ref{thm:OZ:14}. We will show that for triangulations we can relax the restriction on 4-connectedness a bit and still have that the graph is hamiltonian-connected. More specific we will prove the following theorem.

\begin{theorem}\label{thm:sep_tri_share_edge}
Let \(G\) be a 3-connected triangulation. Let \(uv\) be an edge which is contained in all separating triangles of \(G\). Then \(G\) is hamiltonian-connected.
\end{theorem}

As an immediate and easy corollary of this theorem, we get that this property trivially also holds for triangulation with exactly one separating triangle.

\begin{corollary}\label{cor:1septri_hc}
Let \(G\) be a 3-connected triangulation with exactly one separating triangle. Then \(G\) is hamiltonian-connected.
\end{corollary}


\begin{figure}
\begin{center}
\begin{tikzpicture}[thick]

\node [circle, fill, inner sep=1.5pt, label={south west:\(u\)}] (u) at (-2,0) {};
\node [circle, fill, inner sep=1.5pt, label={south east:\(v\)}] (v) at (2,0) {};
\node [circle, fill, inner sep=1.5pt, label={north:\(w_1\)}] (w1) at (0,1) {};
\node [circle, fill, inner sep=1.5pt, label={south:\(w_2\)}] (w2) at (0,-1) {};
\node [circle, fill, inner sep=1.5pt, label={south east:\(z\)}] (z) at (0,0) {};

\draw (u) -- (z) -- (v) -- (w1) -- (u) (v) -- (w2) -- (u) (w1) -- (z) -- (w2);
\draw (u) -- (0,3) -- (v) -- (0,5) -- (u);
\draw[dotted, ultra thick] (0,3.1) -- (0,4.9);

\end{tikzpicture}
\end{center}
\caption{Subdividing the edge shared by all separating triangles.}\label{fig:sep_tri_share_edge}
\end{figure}
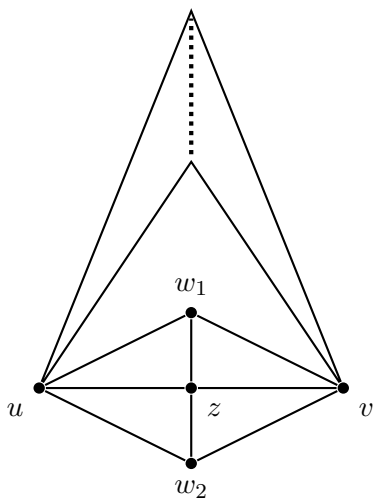

\begin{proof}[Proof of Theorem~\ref{thm:sep_tri_share_edge}]
Let \(G\) be a 3-connected triangulation and let \(uv\) be an edge which is contained in all separating triangles of \(G\). The edge \(uv\) is contained in two facial triangles of \(G\) which we denote by \(uvw_1\) and \(uvw_2\). We obtain the graph \(G'\) by subdividing the edge \(uv\) with the vertex \(z\) and connecting \(z\) with the vertices \(w_1\) and \(w_2\). The graph \(G'\) is shown in Figure~\ref{fig:sep_tri_share_edge}. The modification adds a vertex in each separating triangle and no new separating triangles are created, so \(G'\) is 4-connected. Owing to Theorem~\ref{thm:OZ:14} \(G'\) is 2-edge-hamiltonian-connected.

In order to show that \(G\) is hamiltonian-connected, we have to show for each pair of distinct vertices \(x,y \in V(G)\) that there is a hamiltonian path from \(x\) to \(y\). Assume first that \(\{x,y\}=\{u,v\}\). Since \(G'\) is 2-edge-hamiltonian-connected, it contains a hamiltonian cycle \(C\) through the edges \(uz\) and \(zv\). The path \(C-\{uz, zv\}\) cannot contain the edges \(w_1z\) or \(w_2z\), and therefore forms a hamiltonian path between \(u\) and \(v\) in \(G\).

Assume now that \(\{x,y\}\neq\{u,v\}\). W.l.o.g. we can assume that \(v\notin \{x,y\}\). Since \(G'\) is 2-edge-hamiltonian-connected, we have that \(G' \cup {xy}\) contains a hamiltonian cycle \(C\) through \(xy\) and \(zv\). The cycle \(C\) contains at some point the sequence \(*zv\) where \(* \in \{w_1, w_2, u\}\). We obtain the cycle \(C'\) from \(C\) by replacing the sequence \(*zv\) by the edge \(*v\). The path \(C' - \{xy\}\) is a hamiltonian path from \(x\) to \(y\) in \(G\).
\end{proof}

\section{Decomposition trees with maximum degree \(\Delta(T)>3\)}\label{sec:non_hc}

Jackson and Yu \cite{JY:02} defined a decomposition tree for a 3-connected triangulation as follows: A triangulation with a separating triangle \(S\) can be split into two triangulations: the subgraphs inside and outside of the separating triangle with a copy of \(S\) contained in both. By iteratively applying this procedure to a triangulation with \(k\) separating triangles, we obtain a collection of \(k+1\) triangulations without separating triangles. These 4-connected triangulations form the vertices of the decomposition tree. Two vertices are adjacent if the corresponding pieces share a separating triangle in the original triangulation.

In \cite{Jung:1978} the \emph{scattering number} \(s(G)\) of a graph \(G\) is defined as \[s(G)=\max\{k(G-X)-|X|:X\subseteq V(G), k(G-X)\neq 1\},\]
where \(k(H)\) denotes the number of components of the graph \(H\). In \cite{He:1988} it is noted that \(s(G) \leq -1\) is a necessary condition for \(G\) to be hamiltonian-connected.

Let \(W_n\) be the plane graph obtained by adding a vertex in the center of a cycle of length \(n\) and connecting it to all vertices of the cycle.

\begin{theorem}
For each tree \(T\) with maximum degree \(\Delta(T)>3\), there exists a 3-connected triangulation \(G\) with \(T\) as decomposition tree such that \(G\) is not hamiltonian-connected.
\end{theorem}
\begin{proof}
Let \(v\) be a vertex of \(t\) with degree \(d(v)>3\). Removing \(v\) from \(T\) results in  \(d(v)\) components which we denote by \(T_1,\dots,T_{d(v)}\). By subdiving double wheels with double wheels, we can construct a triangulation \(G_i\) for each of the trees \(T_i\) such that \(G_i\) has \(T_i\) as a decomposition tree. At this point we have no requirements for each triangulation\(G_i\) except that it has \(T_i\) as decomposition tree and that the part corresponding to the vertex of \(T_i\) that neighboured \(v\) in \(T\) has a facial triangle that is not subdivided. Let \(G\) be the graph obtained by subdividing each of the triangular faces of \(W_{d(v)}\) with one of the \(G_i\)'s, and subdividing the outer face with a single vertex. We have that \(k(G - V(W_{d(v)}))=d(v)+1\) and \(|V(W_{d(v)})|=d(v)+1\), so \(s(G)\geq 0\). This implies that \(G\) is not hamiltonian-connected.
\end{proof}

This technique cannot be used to exclude any subcubic tree as the decomposition tree of a hamiltonian-connected triangulation. In \cite{BSVC:2015} it is shown that in a triangulation \(G\) with a subcubic decomposition tree, we have for each splitting set \(S\): \(k(G-S) < |S|\). So the scattering number of \(G\) is at most -1. 

\section{Computational results}\label{sec:computational}

In order to check whether a graph with \(n\) vertices is hamiltonian-connected, we need to check whether there exists a hamiltonian path between \(\frac{n(n-1)}{2}\) pairs of vertices. If two adjacent vertices are connected by a hamiltonian path, then we actually have a hamiltonian cycle, and we can conclude that all vertices that are adjacent on the cycle are connected by a hamiltonian path. This can be used to eliminates several pairs while looking for hamiltonian paths. Below we will see some more results which can further speed up the programs that verify whether a triangulation is hamiltonian-connected.

We will need a theorem proven by Jackson and Yu in \cite{JY:02}.

\begin{theorem}[\cite{JY:02}, (4.2)]\label{thm:JY:02_2edges}
Let \(G\) be a 3-connected triangulation with a decomposition tree of maximum degree at most three. Let \(H\) be a piece of \(G\) corresponding to a vertex of degree at most 2, \(t\) be a facial cycle of both \(H\) and \(G\), and \(V(t) = \{u,v,w\}\). Then \(G\) has a hamiltonian cycle through \(uv\) and \(vw\).
\end{theorem} 

The next corollary immediately follows from this theorem.

\begin{corollary}\label{cor:path_every_edge}
Let \(G\) be a 3-connected triangulation with a path as decomposition tree. Then \(G\) has a hamiltonian cycle through any edge.
\end{corollary}
\begin{proof}
Let \(uv\) be an edge of \(G\). Pick one of the two faces containing the edge \(uv\), and label the third vertex \(w\). Since the maximum degree of the decomposition tree is 2, all conditions for Theorem~\ref{thm:JY:02_2edges} are met, so \(G\) has a hamiltonian cycle through \(uv\) and \(vw\), so certainly through \(uv\).
\end{proof}

If a triangulation has two separating triangles, then the decomposition tree is always a path of length 2. If it has three separating triangles, then the decomposition tree is either a path of length 3 or \(K_{1,3}\).

The following lemma was useful for faster checking whether a graph is hamiltonian-connected, since it allows us to decide on the existence of several other hamiltonian paths based on a single hamiltonian path.

\begin{lemma}
Let \(G\) be a graph with \(n\) vertices. Let \(P\) be a hamiltonian path in \(G\). Let \(x_1,\dots,x_n\) be the sequence of vertices on this path (so in \(P\) we have that \(x_i\) is adjacent to \(x_{i+1}\) for \(1\leq i < n\)). If \(x_i\) (\(1<i\leq n\)) is adjacent to \(x_1\) in \(G\), then there is a hamiltonian path from \(x_{i-1}\) to \(x_n\).
\end{lemma}
\begin{proof}
The path from \(x_{i-1}\) to \(x_n\) is given by \(x_{i-1}\dots x_1x_i\dots x_n\).
\end{proof}

The lemma above is obviously more powerful if the degree of the end points of the paths are large. Therefore we sorted the vertices, so that we first checked for paths between the vertices with the largest degrees. After applying this lemma, we get several hamiltonian paths between new pairs of vertices. For each of these paths, we can again apply this lemma. This also vastly increased the speed of the program. Applying the lemma a third time to the new paths did not really deliver a significant increase in speed.

The following two lemmata do not give any information on the existence of a hamiltonian path between two specific vertices, but reduces the non-existence of such a path to the non-existence of another path in a different triangulation.

\begin{lemma}
Let \(G\) be a triangulation on \(n\) vertices containing \(s\) separating triangles. Let \((u,v,w_1)\) and \((u,v,w_2)\) be two facial triangles of \(G\) such that \(w_1 \nsim w_2\) and \(N(w_1)\cap N(w_2) = \{u,v\}\).
If each pair of adjacent vertices in all triangulations on \(n\) vertices containing at most \(s\) separating triangles are connected by a hamiltonian path, then \(G\) has a hamiltonian path from \(w_1\) to \(w_2\).
\end{lemma}

\begin{proof}
Consider the graph \(G'\) obtained from \(G\) by removing the edge \(uv\) and adding the edge \(w_1w_2\). The graph \(G'\) is a triangulation on \(n\) vertices having at most \(s\) separating cycles, so it contains a hamiltonian path \(P\) from \(w_1\) to \(w_2\). Since all edges of \(G'\) except for \(w_1w_2\) are also contained in \(G\), we have that \(P\) is also a hamiltonian path from \(w_1\) to \(w_2\) in \(G\).
\end{proof}

An edge in a 3-connected simple triangulation is called \emph{reducible} if it is not contained in a separating triangle or a chordless separating quadrangle. The following lemma allows us to skip some pairs of adjacent vertices.

\begin{lemma}
Let \(G\) be a triangulation on \(n\) vertices with a decomposition tree \(D\).
Let \(u\) be a vertex of degree 4, and let \(uv\) be a reducible edge of \(G\).
If each pair of adjacent vertices in all triangulations on \(n-1\) vertices with decomposition tree \(D\) are connected by a hamiltonian path, then \(G\) has a hamiltonian path from \(u\) to \(v\).
\end{lemma}
\begin{proof}
Since the edge \(uv\) is not contained in a separating triangle or a chordless separating quadrangle, the graph obtained by contracting \(uv\) in \(G\) will still have the same decomposition tree, even if \(uv\) is contained in separating quadrangle with a chord. Let \(v, w_1, w_2, w_3\) be the cyclic order of the vertices around \(u\). Let \(G'\) be the triangulation obtained from \(G\) by removing \(u\) and its incident edges and adding the edge \(vw_2\), i.e., \(G'\) is the triangulation obtained by contracting the edge \(uv\). In \(G'\), there is a hamiltonian path \(P\) from \(v\) to \(w_2\). All edges in \(G'\) except the edge \(vw_2\) are also contained in \(G\), so \(P \cup \{uw_2\}\) is a hamiltonian path from \(u\) to \(v\) in \(G\).
\end{proof}

The lemma above gives no new information for triangulations where the decomposition tree is a path, since in that case we already know that adjacent vertices are connected by a hamiltonian path. The following two lemmata prove similar results for triangulations with a path as decomposition tree, but this time for certain vertices at distance two.

\begin{lemma}
Let \(G\) be a triangulation with a path as decomposition tree.
Let \(u\) be a vertex of degree 4, and let \(v_1,v_2,v_3,v_4\) be the cyclic order of the vertices around \(u\).
Let \(uv_1\) be a reducible edge of \(G\).
Then \(G\) has a hamiltonian path from \(v_1\) to \(v_3\).
\end{lemma}
\begin{proof}
Similar to the proof in the previous lemma, we can conclude that the triangulation \(G'\) obtained by contracting the edge \(uv_1\) in \(G\) has a path as decomposition tree. Owing to Theorem~\ref{thm:JY:02_2edges}, \(G'\) has a hamiltonian cycle \(C\) through \(v_1v_2\) and \(v_1v_3\). A hamiltonian path from \(v_1\) to \(v_3\) in \(G\) is \((C\setminus \{v_1v_2, v_1v_3\}) \cup \{vv_1, vv_2\}\).
\end{proof}

\begin{lemma}
Let \(G\) be a triangulation with a path as decomposition tree.
Let \(u\) be a vertex of degree 5, and let \(v_1,v_2,v_3,v_4,v_5\) be the cyclic order of the vertices around \(u\).
Let \(uv_1\) be a reducible edge of \(G\).
Then \(G\) has a hamiltonian path from \(v_1\) to \(v_3\), and from \(v_1\) to \(v_4\).
\end{lemma}
\begin{proof}
We will only give the proof for the hamiltonian path from \(v_1\) to \(v_3\). The other proof is completely analogous.

Similar to the proof in the previous lemma, we can conclude that the triangulation \(G'\) obtained by contracting the edge \(uv_1\) in \(G\) has a path as decomposition tree. Owing to Theorem~\ref{thm:JY:02_2edges}, \(G'\) has a hamiltonian cycle \(C\) through \(v_1v_3\) and \(v_3v_4\). A hamiltonian path from \(v_1\) to \(v_3\) in \(G\) is \((C\setminus \{v_1v_3, v_3v_4\}) \cup \{vv_3, vv_4\}\).
\end{proof}

The program described above was compared with an independent implementation which checks for each pair of vertices whether a hamiltonian cycle exists in the graph obtained by adding a single vertex and connecting it to the two vertices in the pair. It was tested for 3-connected triangulations how many are hamiltonian-connected. Owing to Corollary~\ref{cor:path_every_edge}, only pairs of non-adjacent vertices need to be checked in case of two separating triangles, but no further optimisations were used in this implementation for testing purposes.

The program with the optimisations was used to test 3-connected triangulations with two separating triangles up to 22 vertices. For 22 vertices it was run on a cluster of Intel Xeon E5-2680 CPU's at 2.5 GHz. There were 21 282 658 291 triangulations with two separating triangles and the computations took about 22.7 CPU years. The results of this computation is the following lemma:

\begin{lemma}
On up to 22 vertices all 3-connected triangulations with two separating triangles are hamiltonian-connected.
\end{lemma}

The program with the optimisations was also used to test 3-connected triangulations with three separating triangles up to 21 vertices. For 21 vertices it was run on a cluster of Intel Xeon E5-2680 CPU's at 2.5 GHz. There were 8 751 268 952  triangulations with three separating triangles and the computations took about  6.3 CPU years. The result of this computation is the following lemma:

\begin{lemma}
On up to 21 vertices all 3-connected triangulations with three separating triangles are hamiltonian-connected.
\end{lemma}

We also tested 3-connected triangulations with more than 3 separating triangles, but with specific decomposition trees. We verified hamiltonian-connectedness for triangulations with a path as decomposition tree up to 21 vertices. For 21 vertices it was run on a cluster of Intel Xeon E5-2680 CPU's at 2.5 GHz. There were 10 141 293 048 triangulations with a path as a decomposition tree and the computations took about 5.4 CPU years.
The results of this computation is the following lemma:

\begin{lemma}
On up to 21 vertices all 3-connected triangulations with a path as decomposition tree are hamiltonian-connected.
\end{lemma}

We also verified hamiltonian-connectedness for triangulations with a decomposition tree with maximum degree 3 up to 20 vertices. For 20 vertices it was run on a cluster of Intel Xeon E5-2680 CPU's at 2.5 GHz. There were 23 748 083 814 triangulations with a decomposition tree with maximum degree 3 and the computations took about 8.0 CPU years.
The results of this computation is the following lemma:

\begin{lemma}
On up to 20 vertices all 3-connected triangulations with a decomposition tree with maximum degree 3 are hamiltonian-connected.
\end{lemma}

\section{Conclusion}

It is known that 4-connected triangulations are hamiltonian-connected. We showed that also 3-connected triangulations with only one separating triangle are hamiltonian-connected. If we express this in terms of the decomposition tree, then we have that 3-connected triangulations with a decomposition tree with maximum degree at most 1 are hamiltonian-connected. In order to show bounds on the strongest form of this theorem, we proved that any decomposition tree with maximum degree at least 4 can appear as the decomposition tree of a 3-connected triangulation that is not hamiltonian-connected. This also means that if \(s\geq4\) that there is always a 3-connected triangulation with \(s\) separating triangles that is not hamiltonian-connected. Finally, we presented some computational results which show that all `small' 3-connected triangulations with a decomposition tree with maximum degree at most 3 are hamiltonian-connected.

\section*{Acknowledgements}

The author would like to thank Gunnar Brinkmann for the valuable feedback which drastically shortened the proof of Theorem~\ref{thm:sep_tri_share_edge}. The author would also like to thank Jasper Souffriau for providing him with the programs to determine the decomposition tree of a triangulation which were developed for \cite{BSVC:2015}.

The computational resources (Stevin Supercomputer Infrastructure) and services used in this work were provided by the VSC (Flemish Supercomputer Center), funded by Ghent University, the Hercules Foundation and the Flemish Government -- department EWI.

\end{document}